\documentclass{amsart}
\usepackage[hmargin=30mm, vmargin=30mm]{geometry}

\usepackage[english]{babel}
\usepackage{amsmath,amssymb}
\usepackage{mathrsfs}
\usepackage{dsfont}
\usepackage{amscd}
\usepackage{lmodern}
\usepackage{color}
\numberwithin{equation}{section}
\usepackage{mathtools}

\newtheorem{thmintro}{Theorem}

\newtheorem{corintro}[thmintro]{Corollary}
\newtheorem{theorem}{Theorem}[section]

\newtheorem{lemma}[theorem]{Lemma}
\newtheorem{prop}[theorem]{Proposition}
\theoremstyle{definition}
\newtheorem{definition}[theorem]{Definition}
\newtheorem{remark}[theorem]{Remark}

\newtheorem*{notation}{Notation}

\renewcommand{\hat}{\widehat} 
\newcommand{\bC}{\mathbb{C}}
\newcommand{\C}{\mathbb{C}}
\newcommand{\Z}{\mathbb{Z}}
\newcommand{\R}{\mathbb{R}}
\newcommand{\CC}{\mathbb{C}}

\newcommand{\KK}{\mathbb{K}}
\newcommand{\NN}{\mathbb{N}}

\newcommand{\RR}{\mathbb{R}}
\newcommand{\ZZ}{\mathbb{Z}}
\newcommand{\cH}{\mathcal{H}}

\newcommand{\g}{\mathfrak{g}}
\newcommand{\kk}{\mathfrak{k}}

\newcommand{\hh}{\mathfrak{h}}
\newcommand{\fh}{\mathfrak{h}}

\newcommand{\HH}{\mathcal{H}}
\newcommand{\PP}{\mathcal{P}}
\newcommand{\uu}{\mathfrak{u}}
\newcommand{\fu}{\mathfrak{u}}

\newcommand{\inv}{^{-1}}
\newcommand{\la}{\lambda}
\newcommand{\co}{\colon\thinspace}
\newcommand{\gl}{{\mathfrak g}{\mathfrak l}}
\newcommand{\sll}{{\mathfrak s}{\mathfrak l}}

\DeclareMathOperator{\Aut}{Aut}

\DeclareMathOperator{\sppan}{span}

\DeclareMathOperator{\GL}{GL}

\DeclareMathOperator{\tr}{tr}

\DeclareMathOperator{\der}{der}
\DeclareMathOperator{\su}{sum}

\DeclareMathOperator{\conv}{conv}
\DeclareMathOperator{\End}{End}
\DeclareMathOperator{\mi}{min}
\DeclareMathOperator{\diag}{diag}

\begin{document}

\renewcommand{\proofname}{{\bf Proof}}

\title[Positive energy representations for locally finite split Lie algebras]{Positive energy representations\\ for locally finite split Lie algebras}

\author[Timoth\'ee Marquis]{Timoth\'ee \textsc{Marquis}$^*$}
\address{Department Mathematik, FAU Erlangen-Nuernberg, Cauerstrasse 11, 91058 Erlangen, Germany}
\email{marquis@math.fau.de}
\thanks{$^*$Supported by a Marie Curie Intra-European Fellowship}

\author[Karl-Hermann Neeb]{Karl-Hermann \textsc{Neeb}$^\dagger$}
\address{Department Mathematik, FAU Erlangen-Nuernberg, Cauerstrasse 11, 91058 Erlangen, Germany}
\email{neeb@math.fau.de}
\thanks{$^\dagger$Supported by DFG-grant NE 413/7-2, Schwerpunktprogramm ``Darstellungstheorie"}

%\date{June 2015}

\begin{abstract} 
Let $\g$ be a locally finite split simple complex Lie algebra 
of type $A_J$, $B_J$, $C_J$ or $D_J$ and $\fh \subseteq \g$ be a splitting Cartan subalgebra. 
Fix $D \in \der(\g)$ with $\fh \subseteq \ker D$ (a diagonal derivation). 
Then every unitary highest weight representation 
$(\rho_\lambda, V^\lambda)$ of $\g$ extends to a representation $\tilde\rho_\lambda$ of 
the semidirect product $\g \rtimes \C D$ and we say that 
$\tilde\rho_\lambda$ is a positive energy representation if the 
spectrum of $-i\tilde\rho_\lambda(D)$ is bounded from below. 
In the present note we characterise all pairs $(\lambda,D)$ with $\la$ bounded for which this is the case. 

If $U_1(\cH)$ is the unitary group of Schatten class $1$ on an infinite 
dimensional real, complex or quaternionic Hilbert space and 
$\lambda$ is bounded, then we accordingly obtain a characterisation of 
those highest weight representations $\pi_\lambda$ satisfying the positive energy 
condition with respect to the continuous $\R$-action induced by $D$. 
In this context the representation 
$\pi_\lambda$ is norm continuous and our results imply the remarkable result 
that, for positive energy representations, adding a suitable inner derivation to $D$, 
we can achieve that the minimal eigenvalue of $\tilde\rho_\lambda(D)$ is $0$ 
(minimal energy condition). The corresponding pairs $(\lambda,D)$ satisfying 
the minimal energy condition are rather easy to describe explicitly. 
\end{abstract}

\maketitle

\section{Introduction}
Locally finite split Lie algebras are natural infinite-dimensional generalisations of finite-dimensional Lie algebras. More precisely, a Lie algebra $\g$ over a field $\KK$ of characteristic zero is called \emph{split} if it has a root decomposition $\g=\hh+\sum_{\alpha\in\Delta}{\g_{\alpha}}$ with respect to some maximal abelian subalgebra $\hh$. It is moreover \emph{locally finite} if every finite subset of $\g$ generates a finite-dimensional subalgebra. Such a Lie algebra $\g$ possesses a generalised Levi decomposition (see \cite{St99}), whose Levi factor is an $\hh$-invariant semisimple Lie algebra. In turn, the infinite-dimensional locally finite split simple Lie algebras have been classified (see \cite{NeSt01}), and can be realised as subalgebras of the algebra $\gl(J,\KK)$ of $J\times J$ matrices with only finitely many nonzero entries, for some infinite set $J$. They fall into four distinct families of isomorphism classes, parametrised by the \emph{locally finite root systems} of type $A_J$, $B_J$, $C_J$ and $D_J$ (see \cite{LN04}).

On the other hand, a locally finite split simple Lie algebra $\g$ is the directed union of its finite-dimensional simple subalgebras (\cite[Section~V]{St99}). Unitary highest weight representations for such Lie algebras over $\KK=\CC$ were studied in \cite{Ne98}. By \cite[Section~VIII]{St99}, $\g$ carries an antilinear involutive antiautomorphism $X\mapsto X^*$ such that $\g_{\RR}=\{X\in\g \ | \ X^*=-X\}$ is a compact real form, that is, $\g_{\RR}$ is a union of finite dimensional compact Lie algebras. A $\g$-module $V$ is then called \emph{unitary} if it carries a contravariant positive definite hermitian form $\langle\cdot,\cdot\rangle$, in the sense that $\langle X.v,w\rangle=\langle v,X^*.w\rangle$ for all $v,w\in V$, $X\in\g$. A $\g$-module $V=V^{\la}$ is called a \emph{highest weight module} with highest weight $\la\in\hh^*$ (with respect to some positive system $\Delta_+\subseteq\Delta$ of roots) if it is generated by some primitive element, that is, by some $\hh$-weight vector $v\in V$ with weight $\la$ such that $\g_{\alpha}.v=\{0\}$ for all $\alpha\in\Delta_+$. Unitary highest weight modules $V^{\la}$ for $\g$ were classified in \cite{Ne98}, and correspond to dominant integral weights $\la$. Moreover, if $W\leq\GL(\hh^*)$ denotes the Weyl group of $\g$, the set $\PP_{\la}$ of $\hh$-weights on $V^{\la}$ is given by $\PP_{\la}=\conv(W.\la)\cap (\la+\ZZ[\Delta])$ (\cite[Theorem~I.11]{Ne98}).

In this paper, we characterise the unitary highest weight representations of $\g$ 
satisfying the following ``positive energy condition''. 
A skew-hermitian derivation $D\in\der(\g)$ of $\g$ is called \emph{graded} if it annihilates $\hh$, in which case $D$ preserves the root spaces of $\g$. Such a derivation is then described by a character $\chi\co\ZZ[\Delta]\to\RR$ such that $D(x_{\alpha})=i\chi(\alpha)x_{\alpha}$ for all $x_{\alpha}\in\g_{\alpha}$, $\alpha\in\Delta$. Any unitary highest weight representation $\rho_{\la}\co\g\to \End(V^{\la})$ of $\g$ may be extended to a representation $\widetilde{\rho}_{\la}\co \g\rtimes \CC D\to \End(V^{\la})$ by setting $\widetilde{\rho}_{\la}(D)v_{\mu}=i\chi(\mu-\la)v_{\mu}$ for any $\mu\in\PP_{\la}$ and any weight vector $v_{\mu}\in V^{\la}$ of weight $\mu$. The representation $\widetilde{\rho}_{\la}$ is called a \emph{positive energy representation} if the spectrum of $H:=-i\widetilde{\rho}_{\la}(D)$ 
is bounded from below. We call $H$ a \emph{Hamiltonian} of the representation 
$\rho_\lambda$.

In view of the above description of $\PP_{\la}$, the positive energy condition can then be rewritten as
$$\inf\chi(W.\la-\la)>-\infty.$$
Let $J$ be an infinite set such that $\Delta$ is of one of the types $A_J$, $B_J$, $C_J$ or $D_J$. Then there exists some basis $(e_j)_{j\in J}$ of $\hh$ (respectively 
a one-dimensional extension of $\hh$) indexed by $J$ such that the set of \emph{coroots} $\Delta^{\vee}=\{\alpha^{\vee} \ | \ \alpha\in\Delta\}$ (see \cite[Section~I]{St99}) is contained in $\sppan_{\ZZ}(e_j)_{j\in J}$, while $\Delta$ is contained in the $\ZZ$-span of the linearly independent system $(\epsilon_j)_{j\in J}\subseteq\hh^*$ defined by $\epsilon_j(e_k)=\delta_{jk}$ for all $j,k\in J$ (see \cite[Section~I and Theorem~IV.6]{NeSt01}). 

We recall that the weight $\la\in\hh^*$ is called \emph{integral} if $\la(\alpha^{\vee})\in\ZZ$ for all $\alpha\in\Delta$. This implies in particular that $\la$ is \emph{discrete}, in the sense that $\{\la_j:=\la(e_j) \ | \ j\in J\}$ is a discrete subset of $\RR$. We moreover call $\la$ \emph{bounded} if $\sup_{j\in J}|\la_j|<\infty$. Finally, a character $\chi\co\ZZ[\Delta]\to\RR$ is said to be \emph{summable} if
it is the restriction of a homomorphism 
$\widetilde{\chi}\co\sppan_{\ZZ}(\epsilon_j)_{j\in J}\to\RR$
 satisfying $$\sum_{j\in J}{|\widetilde{\chi}(\epsilon_j)|}<\infty.$$
We can now state the main result of this paper, which provides a characterisation of the positive energy highest weight representations of $\g$ with bounded highest weight.

\begin{thmintro}\label{thm:mainintro}
Let $(\g,\hh)$ be a locally finite split simple Lie algebra with root system $\Delta$ and Weyl group $W\leq\GL(\hh^*)$. Let $\la\in\hh^*$ be discrete and bounded. Then for a character $\chi\co\ZZ[\Delta]\to\RR$, the following are equivalent:
\begin{enumerate}
\item $\inf\chi(W.\la-\la)>-\infty.$
\item $\chi=\chi_{\mi}+\chi_{\su}$ for some ``minimal energy'' 
character $\chi_{\mi}$, satisfying $\inf\chi_{\mi}(W.\la-\la)=0$, and some summable character $\chi_{\su}$.
\end{enumerate}
\end{thmintro}

\noindent
We also provide a description of ``minimal energy'' characters (see Proposition~\ref{prop:minimality} and Remark~\ref{remark:min_energy_DJ}). 
In \cite[Corollary~3.2]{Convexhull} a similar description of the 
``minimal energy'' characters was obtained by Coxeter geometry. 
A priori, these can be described much more easily than the ``positive energy'' 
characters and the main point of the theorem is that it reduces the latter 
problem to the former. The proof of Theorem~\ref{thm:mainintro} 
is given in Section~\ref{section:proof_main_thmintro} below.

\medskip

The assumption that the highest weight $\la$ be bounded in Theorem~\ref{thm:mainintro} is motivated by the study of positive energy (projective) unitary representations of the corresponding Lie groups, 
which we now briefly review. We then state a corollary of Theorem~\ref{thm:mainintro} in this context.

The real form $\g_{\RR}$ of the locally finite split simple Lie algebra $\g$ is endowed with an invariant scalar product $\langle\cdot,\cdot\rangle$, in the sense that $\langle [X,Y],Z\rangle =\langle X,[Y,Z]\rangle$ for all $X,Y,Z\in\g_{\RR}$. For instance, if $\g=\sll(J,\CC)$ is the subalgebra of $\gl(J,\CC)$ of traceless matrices, it is given by $\langle X,Y\rangle=\tr(XY^*)$. 
The Hilbert space completion $\kk$ of $\g_{\RR}$ is then a so-called \emph{Hilbert--Lie algebra}, that is, a real Lie algebra and a real Hilbert space with compatibility of the two structures given by the invariance of the scalar product. For $\g=\sll(J,\CC)$, the corresponding completion is the space $\uu_2(\HH)$ of skew-symmetric Hilbert-Schmidt operators on the complex Hilbert space $\HH = \ell^2(J,\bC)$. 

By a theorem of Schue (\cite{Schue61}), Hilbert--Lie algebras decompose into an orthogonal direct sum of simple ideals (and center). Moreover, each simple infinite-dimensional Hilbert--Lie algebra is isomorphic to $\uu_2(\HH)$, for some infinite-dimensional real, complex or quaternionic Hilbert space $\HH$. These are, in turn, classified by the locally finite root systems (see \cite[Examples~C.4,5,6]{Ne12}).

Going back to our example $\g=\sll(J,\CC)$, we let $\rho_{\la}\co\g\to \End(V^{\la})$ be, as before, a unitary highest weight representation of $\g$ with highest weight $\la$. Then, under the assumption that $\la$ is bounded, the restriction of $\rho_{\la}$ to $\g_{\RR}$ extends to a continuous unitary representation $\rho_{\la}\co \uu_1(\HH)\to \uu(\HH_{\la})$ with $\left\Vert \rho_{\la} \right\Vert\leq \sup_{j\in J}|\la_j|$, where $\uu_1(\HH)\subseteq \uu_2(\HH)$ denotes the Banach space of skew-hermitian trace-class operators on $\HH$ and $\HH_{\la}$ is the Hilbert space completion of $V^{\la}$ (see \cite[Proposition~III.7]{Ne98}). Moreover, if $\la_j\in\ZZ$ for all $j\in J$, then $\rho_{\la}$ exponentiates to a holomorphic representation $$\widehat{\rho}_{\la}\co U_1(\HH)\to U(\HH_{\la})$$ from $U_1(\HH)=\GL(\HH)\cap ({\mathbf 1} + \uu_1(\HH))$ to the unitary group $U(\HH_{\la})$.

Assume now that the Lie group $G=U_1(\HH)$ is endowed with a continuous $\RR$-action, given by a homomorphism $\alpha\co \RR\to\Aut(G):t\mapsto\alpha_t$. If $U\co\RR\to U(\HH_{\la}):t\mapsto U_t$ is a unitary representation of $\RR$ on $\HH_{\la}$, then the map 
$$\pi_{\la}\co G\rtimes_{\alpha}\RR\to  U(\HH_{\la}):(g,t)\mapsto \rho_{\la}(g)U_t$$
is called a \emph{covariant unitary representation} of $(G,\RR,\alpha)$ if it defines a unitary representation on $\HH_{\la}$ of the semi-direct product $G\rtimes_{\alpha}\RR$. This representation is said to be of \emph{positive energy} if the spectrum of the corresponding Hamiltonian $H:=-i\frac{d}{dt}|_{t=0}U_t$ is bounded below. 

In our setting, one can show that $\alpha$ must be of the form
\begin{equation}
  \label{eq:action}
\alpha_t(g)=e^{itA}ge^{-itA} 
 \end{equation}
for some self-adjoint bounded operator $A\in B(\HH)$. A sufficient condition for $\rho_{\la}$ to extend to a covariant representation of $(G,\RR,\alpha)$ is the diagonalisability of $A$: in this case, choosing the orthonormal basis $(e_j)_{j\in J}$ of $\HH$ so that $Ae_j=d_je_j$ for all $j\in J$, for some $d_j\in\RR$, one gets a covariant representation $\pi_{\la}$ as above by setting 
$$U_tv_{\la}=v_{\la}\quad\textrm{and}\quad U_tv_{\mu}=e^{it\chi(\mu-\la)}v_{\mu}$$
for any $t\in \RR$, $\mu\in\PP_{\la}$ and any $\mu$-weight vector $v_{\mu}\in\HH_{\la}$, where $\chi\co\ZZ[\Delta]\to\RR$ is the character induced by the assignment $\epsilon_j\mapsto d_j$. Comparing this situation with Theorem~\ref{thm:mainintro}, we see that the decomposition $\chi=\chi_{\mi}+\chi_{\su}$ in this theorem corresponds to a decomposition $A=A_{\mi}+A_{\su}$ of $A$ as a sum of two commuting (simultaneously diagonalisable) operators $A_{\mi},A_{\su}\in B(\HH)$ such that $iA_{\su}\in \uu_1(\HH)$, and such that $A_{\mi}$ yields a \emph{minimal energy representation} $\pi_{\la}$, in the sense that the corresponding Hamiltonian $H_{\mi}$ is non-negative (with eigenvalue $0$ on the highest weight vector $v_{\la}$). Writing $\alpha^{\mi}$ and $\alpha^{\su}$ for the $\RR$-actions on $G$ induced by $A_{\mi}$ and $A_{\su}$ respectively, this implies that $\alpha_t$ differs only from $\alpha_t^{\mi}$ by an inner automorphism $\alpha_t^{\su}$ of $G$ commuting with $\alpha_t$ and $\alpha_t^{\mi}$ ($t\in \RR$): in this case, we will say that the corresponding covariant representations of $G\rtimes_{\alpha}\RR$ and $G\rtimes_{\alpha^{\min}}\RR$ are \emph{similar}. Thus Theorem~\ref{thm:mainintro} has the following corollary:

\begin{corintro}\label{corintro:U1}
Let $G=U_1(\HH)$ be endowed with a continuous $\RR$-action $\alpha\co \RR\to\Aut(G)$, given by $\alpha_t(g)=e^{itA}ge^{-itA}$ for some self-adjoint bounded operator $A\in B(\HH)$. Assume that $A$ is diagonalisable. Then every positive energy covariant unitary representation $\pi_{\la}\co G\rtimes_{\alpha}\RR\to  U(\HH_{\la})$ of $(G,\RR,\alpha)$ as defined above is similar to a minimal energy representation.
\end{corintro}

Finally, one can show that the highest weight representation $\widehat{\rho}_{\la}\co U_1(\HH)\to U(\HH_{\la})$ extends to a projective unitary representation $U_2(\HH)\to PU(\HH_{\la})$ of the Hilbert--Lie group 
\[ U_2(\HH)=\GL(\HH)\cap ({\mathbf 1} + \uu_2(\HH)), \] 
and hence to a unitary representation on $\HH_{\la}$ of a central extension 
$\hat U_2(\cH)$ of $U_2(\HH)$, given at the Lie algebra level by the cocycle 
$\omega(X,Y)=i\lambda([X,Y])$ for $X,Y\in\uu_2(\HH)$, where 
we extend $\lambda$ in a natural way to a continuous linear functional on 
$\fu_1(\cH)$ which contains $[X,Y]$ (see \cite{Ne15}). 
Since the continuous $\R$-action \eqref{eq:action} on 
$U_2(\cH)$ lifts canonically to the central extension, 
the corresponding notion of positive energy for the associated 
projective covariant unitary representations of $U_2(\HH)$ 
is also described by Corollary~\ref{corintro:U1}.

The situation discussed in this paper is a model case in which a rather 
detailed analysis of the positive energy condition can be carried out. For every 
triple $(G,\R, \alpha)$, where $G$ is a Lie group and 
$\alpha \colon \R \to \Aut(G)$ defines a continuous action, it is a challenging 
natural problem to determine the irreducible positive energy representations 
$(\pi,\cH)$ of the topological group $G^\sharp := G \rtimes_\alpha \R$. As a consequence 
of the Borchers--Arveson Theorem (\cite[Theorem~3.2.46]{BR02}), for any such representation, 
the restriction $\rho := \pi|_G$ is irreducible 
(see \cite[Theorem~2.5]{Ne14}) and the Hamiltonian of the 
extension to $G^\sharp$ is uniquely determined up to an additive constant 
determining the minimal energy level. Given $\alpha$, the set of irreducible positive energy 
representations of $G^\sharp$ can therefore be considered as a subset 
$\hat G_\alpha$ of the set $\hat G$ of equivalence classes of irreducible unitary 
representations of $G$ and one would like to determine this subset as explicitly 
as possible. In this paper this task is carried out for the subset 
$\hat G_{hw}$ of ``highest  weight representations'' of $G = U_1(\cH)$ in the 
case where $\alpha$ is given by conjugation with diagonal operators.
Here Corollary~\ref{corintro:U1} achieves in the Lie algebra context 
something similar as the 
Borchers--Arveson Theorem which also reduces the study of positive 
energy representations to minimal energy representations. 

One can even show that the representation of the centrally extended group 
$\hat U_2(\cH)$ extends to a group $\hat U_{\rm res}$ 
containing a copy of the centraliser $D$ 
of the diagonal operator $\diag((\lambda_j)_{j \in J})$ in $U(\cH)$ 
(see \cite[Theorem~VII.18]{Ne04} for the $A_J$-case). 
Here $D$ is a finite product $\prod_{m \in \Z} U(\ell^2(J_m))$ 
of full unitary groups where the factors correspond 
to the subsets $J_m := \{ j \in J \ | \ \lambda_j = m\}$. 
The corresponding extension $\hat\pi$ to $\hat U_{\rm res}$ 
is a unitary Lie group representation for which one would like to understand 
the convex cone of all elements $X \in \fu_{\rm res}$ in the Lie algebra for which 
the operator $-id\hat\pi(X)$ is positive. If $X$ is diagonal, 
 this problem is solved by Theorem~\ref{thm:mainintro} if we 
put $\chi(\epsilon_j) := X_{jj}$, but the general case requires refined 
information on convex hulls of adjoint 
orbits (see \cite{Ne10} for similar problems). We plan to address this 
issue in a separate paper because it is of a functional analytic flavour,
whereas the present paper is purely algebraic.

\begin{notation}
Throughout this paper, we denote by $\NN=\{1,2,\dots\}$ the set of positive natural numbers.
\end{notation}

\section{Preliminaries}

\subsection{Locally finite root systems}\label{subsection:preliminaries}
Let $J$ be an infinite set and let $V:=\RR^{(J)}\subseteq \overline{V}:=\RR^J$ be the free vector space over $J$, with canonical basis $\{e_j \ | \ j\in J\}$ and standard scalar product $(e_j,e_k)=\delta_{jk}$. In the dual space $V^*\cong \RR^{J}$, we consider the linearly independent system $\{\epsilon_j:=e_j^* \ | \ j\in J\}$ defined by $\epsilon_j(e_k)=\delta_{jk}$.

Any infinite irreducible (possibly non-reduced) locally finite root system $\Delta$ can be described inside $V^*$ for some suitable set $J$, and is of one of the following types (\cite[\S 8]{LN04}):
\begin{equation*}
\begin{aligned}
A_J&:=\{\epsilon_j-\epsilon_k \ | \ j,k\in J, \ j\neq k\},\\
B_J&:=\{\pm \epsilon_j, \pm(\epsilon_j\pm\epsilon_k) \ | \ j,k\in J, \ j\neq k\},\\
C_J&:=\{\pm 2\epsilon_j, \pm(\epsilon_j\pm\epsilon_k) \ | \ j,k\in J, \ j\neq k\},\\
D_J&:=\{\pm(\epsilon_j\pm\epsilon_k) \ | \ j,k\in J, \ j\neq k\},\\
BC_J&:=\{\pm \epsilon_j, \pm 2\epsilon_j, \pm(\epsilon_j\pm\epsilon_k) \ | \ j,k\in J, \ j\neq k\}.
\end{aligned}
\end{equation*}
For $X\in\{A,B,C,D,BC\}$, we will write $\Delta(X_J)$ for the above locally finite root system of type $X_J$. Note that the root systems of type $A_J$, $B_J$, $C_J$ and $D_J$ are reduced, whereas $\Delta(BC_J)$ is non-reduced.

\subsection{The Weyl group of $\Delta$}\label{subsection:Weyl_group_Delta}
Let $S_J$ denote the symmetric group on $J$, which we view as a subgroup of $\GL(\overline{V})$ with $w\in S_J$ acting as $w(e_j):=e_{w(j)}$. Given a permutation $w\in S_J$ with fixed-point set $I\subseteq J$, we call the set $J\setminus I$ the \emph{support} of $w$. We denote by $S_{(J)}\leq  S_J$ the subgroup of restricted permutations, namely, the set of $w\in S_J$ with finite support. Note that $S_{(J)}\leq \GL(\overline{V})$ stabilises $V$; we will also view $S_{(J)}$ as a subgroup of $\GL(V)$.

We next view $\{\pm 1\}^{J}\subset \RR^J$ as a subgroup of $\GL(\overline{V})$, acting by (componentwise) left multiplication: $\sigma(e_j)=\sigma_je_j$ for $\sigma=(\sigma_j)_{j\in J}\in \{\pm 1\}^{J}$. 
Given some $\sigma=(\sigma_j)_{j\in J}\in \{\pm 1\}^{J}$, we call the subset $I=\{j\in J \ | \ \sigma_j=-1\}$ of $J$ the \emph{support} of $\sigma$. We denote by $\{\pm 1\}^{(J)}$ the set of all $\sigma\in\{\pm 1\}^{J}$ with finite support. Again, we may also view $\{\pm 1\}^{(J)}$ as a subgroup of $\GL(V)$. Finally, we let $\{\pm 1\}^{(J)}_2$ denote the index $2$ subgroup of $\{\pm 1\}^{(J)}$, whose elements have a support of even cardinality.

Let $X\in\{A,B,C,D,BC\}$. We denote by $W=W(X_J)$ the Weyl group corresponding to $\Delta(X_J)$, which we view as a subgroup of $\GL(V)$ or $\GL(\overline{V})$. 
We then have the following descriptions (\cite[\S 9]{LN04}):
\begin{equation*}
\begin{aligned}
W(A_J)&=S_{(J)},\\
W(B_J)&=W(C_J)=W(BC_J)=S_{(J)}\ltimes \{\pm 1\}^{(J)},\\
W(D_J)&=S_{(J)}\ltimes \{\pm 1\}_2^{(J)}.
\end{aligned}
\end{equation*}

\subsection{The positive energy condition}
Let $X\in\{A,B,C,D,BC\}$ and set $W=W(X_J)$. Fix some tuples $\la=(\la_j)_{j\in J}\in\RR^{J}$ and $\chi=(d_j)_{j\in J}\in\RR^{J}$. 

\begin{definition}\label{definition:PEC}
We say that the triple $(J,\la,\chi)$ \emph{satisfies the positive energy condition (PEC) for $W$} if the set $\la(W.\chi-\chi)$ is bounded from below. Here, we view $\la$ as the linear functional
$$\la\co V\to\RR: e_j\mapsto \la_j,$$
and $W.\chi-\chi$ as a subset of $V$, by writing $\chi$ as $\chi=\sum_{j\in J}{d_je_j}\in \overline{V}$. More precisely, recall from \S\ref{subsection:Weyl_group_Delta} that any element of $W$ may be written as a product $\sigma w\inv$ for some $\sigma\in\{\pm 1\}^{(J)}$ and some $w\in S_{(J)}$. Then
\begin{equation*}
\sigma w\inv.\chi-\chi =\sum_{j\in J}{(\sigma_{w\inv(j)}d_je_{w\inv(j)}-d_je_j)}=\sum_{j\in J}{(\sigma_{j}d_{w(j)}-d_j)e_j}\in V.
\end{equation*}
\end{definition}

Note that for any $\sigma\in\{\pm 1\}^{(J)}$ and $w\in S_{(J)}$, we have
\begin{equation}\label{eqn:basic_loc_fin}
\la(\sigma w\inv.\chi-\chi)=\sum_{j\in J}{\la_j(\sigma_{j}d_{w(j)}-d_j)}.
\end{equation}
In particular, given two disjoint finite subsets $\{i_1,i_2,\dots,i_k\}$ and $\{j_1,j_2,\dots,j_k\}$ of $J$, the product $w$ of the $k$ transpositions $\tau_1,\dots,\tau_k\in S_{(J)}$, where $\tau_s$ interchanges $i_s$ and $j_s$ ($s=1,\dots,k$), is an element of $S_{(J)}$ and we have
\begin{equation}\label{eqn:prod_transp}
\la(w.\chi-\chi)=\la(w\inv.\chi-\chi)=\sum_{s=1}^k{(\la_{j_s}-\la_{i_s})(d_{i_s}-d_{j_s})}.
\end{equation}

We record for future reference the following so-called \emph{rearrangement inequality}.  
\begin{lemma}
Let $a_1\leq a_2\leq\dots\leq a_n$ and $b_1\leq b_2\leq \dots\leq b_n$ be two non-decreasing sequences of real numbers. Let also $(c_1,\dots,c_n)$ be a permutation of $(b_1,\dots,b_n)$. Then
$$\sum_{i=1}^{n}{a_ib_i}\geq\sum_{i=1}^{n}{a_ic_i}\geq\sum_{i=1}^{n}{a_ib_{n+1-i}}.$$
\end{lemma}

\section{A few definitions and notations}
Fix some set $J$, as well as two tuples $\la=(\la_j)_{j\in J}\in\RR^J$ and $\chi=(d_j)_{j\in J}\in\RR^J$. 
We define the functions $$D\co J\to \RR:j\mapsto d_j \quad\textrm{and}\quad\Lambda\co J\to\RR:j\mapsto\la_j.$$
We set 
$$m_{\min}=\inf \Lambda(J)\in\RR\cup\{-\infty\}\quad\textrm{and}\quad m_{\max}=\sup \Lambda(J)\in\RR\cup\{\infty\}.$$
We call the triple $(J,\la,\chi)$ \emph{nontrivial} if $\la$ is non-constant, so that $m_{\min}\neq m_{\max}$.

For each $n\in\RR$, we set $$J_n:=\Lambda\inv(n),$$ so that $J=\bigcup_n{J_n}$. Given some $r\in\RR$, we also define the sets 
$$J_n^{>r}=\{j\in J_n \ | \ d_j>r\}\quad\textrm{and}\quad J_n^{<r}=\{j\in J_n \ | \ d_j< r\}.$$
Finally, we let $\Lambda^{\infty}(J)$ denote the set of all $m\in\Lambda(J)$ such that $J_m$ is infinite, and we set
$$\overline{\Lambda^{\infty}}(J):=\big(\Lambda^{\infty}(J)\cup\{m_{\min},m_{\max}\}\big)\cap\RR.$$

\begin{definition}\label{def:locally_bounded}
We call $(J,\la,\chi)$ \emph{essentially bounded} if for all $m\in\Lambda(J)$, the following two conditions are satisfied:
\begin{enumerate}
\item
If $m\neq m_{\max}$, then $D(J_m)$ is bounded below.
\item
If $m\neq m_{\min}$, then $D(J_m)$ is bounded above.
\end{enumerate}
\end{definition}

\begin{definition}\label{definition:accumulation_point}
Given a subset $I$ of $J$, we call $r\in\RR$ an \emph{accumulation point for $I$} if either $r$ is an accumulation point for $D(I)$, or if $D(I')=\{r\}$ for some infinite subset $I'\subseteq I$.
\end{definition}

\begin{definition}\label{definition:min_max_accumulation}
Assume that $(J,\la,\chi)$ is essentially bounded and nontrivial. Let $m\in \overline{\Lambda^{\infty}}(J)$. 
\begin{enumerate}
\item
If $m\neq m_{\min},m_{\max}$, then $D(J_m)$ is bounded, and hence $J_m$ possesses an accumulation point. In this case, we let $r^{\min}_m$ and $r^{\max}_m$ respectively denote the minimal and maximal accumulation points of $J_m$.
\item
If $m=m_{\min}$, then $D(J_m)$ is bounded below. If $J_m$ has an accumulation point, we let $r^{\min}_m$ denote the minimal one. Otherwise, we set $r^{\min}_m=\infty$.
\item
If $m=m_{\max}$, then $D(J_m)$ is bounded above. If $J_m$ has an accumulation point, we let $r^{\max}_m$ denote the maximal one. Otherwise, we set $r^{\max}_m=-\infty$.
\end{enumerate}
\end{definition}

The following lemma, which easily follows from the definitions, provides the geometric picture to be kept in mind for the rest of this paper.
\begin{lemma}\label{lemma:r+m}
Assume that $(J,\la,\chi)$ is essentially bounded and nontrivial. Let $m\in\overline{\Lambda^{\infty}}(J)$.
\begin{enumerate}
\item
If $m\neq m_{\min},m_{\max}$, then for each $\epsilon>0$, there is some finite subset $I_{\epsilon}\subset J_m$ such that $$D(J_m\setminus I_{\epsilon})\subseteq [r^{\min}_m-\epsilon,r^{\max}_m+\epsilon].$$
\item
If $m=m_{\min}$ and $r^{\min}_m\neq\infty$, then for each $\epsilon>0$, there is some finite subset $I_{\epsilon}\subset J_m$ such that 
$$D(J_m\setminus I_{\epsilon})\subseteq [r^{\min}_m-\epsilon,\infty[.$$
\item
If $m=m_{\max}$ and $r^{\max}_m\neq-\infty$, then for each $\epsilon>0$, there is some finite subset $I_{\epsilon}\subset J_m$ such that 
$$D(J_m\setminus I_{\epsilon})\subseteq \thinspace]\!-\infty,r^{\max}_m+\epsilon].$$
\end{enumerate}
\end{lemma}

Before proceeding with the study of the PEC, we need to introduce one more concept. 
\begin{definition}\label{definition:summable}
Given $r\in\RR$, we call a subset $I_r$ of $J$ of the form $J_n^{<r}$ or $J_n^{>r}$ \emph{summable} (with respect to $(J,\la,\chi)$) if
$$\Sigma(I_r):=\sum_{j\in I_r}{|d_j-r|}<\infty.$$
\end{definition}

\section{Consequences of the PEC for $W(A_J)$}
In this section, we fix some infinite set $J$ and some tuples $\la=(\la_j)_{j\in J}\in\RR^{J}$ and $\chi=(d_j)_{j\in J}\in\RR^{J}$, and we assume that $(J,\la,\chi)$ satisfies the PEC for $W=W(A_J)$. 

\begin{lemma}\label{lemma 1}
Let $m,n\in\RR$ with $m<n$.
\begin{enumerate}
\item
If $J_m^{<r}$ and $J_n^{>r}$ are both nonempty for some $r\in\RR$, then $D(J_m^{<r})$ and $D(J_n^{>r})$ are bounded. 
\item
If $J_m$ and $J_n$ are both nonempty, then $D(J_m)$ (resp. $D(J_n)$) is bounded below (resp. above).
\end{enumerate}
\end{lemma}
\begin{proof}
To prove (1), assume for a contradiction that $D(J_m^{<r})$ is unbounded, and for each $k\in\NN$, let $j_k\in J_m^{<r}$ such that $d_{j_k}<-k$. Pick any $j\in J_n^{>r}$. Consider for each $k\in\NN$ the transposition $\tau_k\in W$ exchanging $j$ and $j_k$. It then follows from (\ref{eqn:prod_transp}) that $$\la(\tau_k.\chi-\chi)=(\la_j-\la_{j_k})(d_{j_k}-d_j)=(n-m)(d_{j_k}-d_j)<-(n-m)(k+d_j)$$ for all $k\in\NN$, contradicting the PEC for $W$. The proof that $D(J_n^{>r})$ is bounded is similar.

We now turn to the proof of (2). Assume that $J_m$ and $J_n$ are nonempty. We prove that $D(J_m)$ is bounded below, the proof for $D(J_n)$ being similar. Let $r\in\RR$ be such that $J_n^{>r}$ is nonempty. If $J_m^{<r}$ is empty, then $\inf D(J_m)\geq r$. If $J_m^{<r}$ is nonempty, then $D(J_m^{<r})$ is bounded by (1), and hence $D(J_m)$ is bounded below, as desired.
\end{proof}

\begin{lemma}\label{lemma 2}
Let $m,n\in\RR$ with $m<n$, and let $r\in\RR$. Assume that $J_n^{>r}$ is infinite. Then $J_m^{<t}$ is finite for all $t<r$.
\end{lemma}
\begin{proof}
Assume for a contradiction that $J_m^{<t}$ is infinite for some $t<r$. Let $\{i_1,i_2,\dots\}$ (resp. $\{j_1,j_2,\dots\}$) be an infinite countable subset of $J_m^{<t}$ (resp. $J_n^{>r}$). For each $k\in\NN$, let $w_k\in W$ be the product of the $k$ transpositions $\tau_1,\dots,\tau_k$, where $\tau_s$ interchanges $i_s$ and $j_s$ ($s\in\NN$). It then follows from (\ref{eqn:prod_transp}) that $$\la(w_k.\chi-\chi)=\sum_{s=1}^{k}{(\la_{j_s}-\la_{i_s})(d_{i_s}-d_{j_s})}=(n-m)\sum_{s=1}^{k}{(d_{i_s}-d_{j_s})}<-(n-m)(r-t)k$$ for all $k\in\NN$, contradicting the PEC for $W$.
\end{proof}

\begin{lemma}\label{lemma 3}
Let $m,n\in\RR$ with $m<n$. Then there exists at most one $r\in\RR$ such that $J_m^{<r}$ and $J_n^{>r}$ are both infinite.
\end{lemma}
\begin{proof}
Assume that $J_m^{<r}$ and $J_n^{>r}$ are both infinite for two different values of $r$, say $r_1<r_2$. Using Lemma~\ref{lemma 2} with $r=r_2$ and $t=r_1$ then yields the desired contradiction.
\end{proof}

\begin{lemma}\label{lemma:extremal_values}
If $(J,\la,\chi)$ satisfies the PEC for $W(A_J)$, it is essentially bounded. 
\end{lemma}
\begin{proof}
Let $m\in\Lambda(J)$. If $m\notin \Lambda^{\infty}(J)$, then $D(J_m)$ is finite, hence bounded, and there is nothing to prove. Assume now that $m\in\Lambda^{\infty}(J)$. If $m\neq m_{\max}$, then there is some $n\in\Lambda(J)$ such that $m<n$. Since $J_{n}\neq \emptyset$, Lemma~\ref{lemma 1}(2) implies that $D(J_m)$ is bounded below, as desired. Similarly, if $m\neq m_{\min}$, so that there exists some $n\in\Lambda(J)$ with $m>n$, Lemma~\ref{lemma 1}(2) implies that $D(J_m)$ is bounded above, proving the claim.
\end{proof}

\begin{prop}\label{prop:3cases}
Assume that $(J,\la,\chi)$ satisfies the PEC for $W=W(A_J)$. Let $m,n\in \overline{\Lambda^{\infty}}(J)$ be such that $m<n$. Then one of the following assertions holds.
\begin{enumerate}
\item
$r^{\max}_n<r^{\min}_m$. In this case, there is some $r\in\RR$ such that $J_m^{<r}$ and $J_n^{>r}$ are both finite. 
\item
$r^{\max}_n=r^{\min}_m$. In this case, $J_m^{<r^{\min}_m}$ and $J_n^{>r^{\max}_n}$ are both summable. 
\end{enumerate}
\end{prop}
\begin{proof}
Note first that $(J,\la,\chi)$ is essentially bounded by Lemma~\ref{lemma:extremal_values} and nontrivial by hypothesis, so that $r^{\min}_m$ and $r^{\max}_n$ are defined. 

If $m=m_{\min}$ and $r^{\min}_m=\infty$, so that $D(J_m)$ is bounded below and $J_m$ has no accumulation point, then $J_m^{<r}$ is finite for any $r\in\RR$. Since in addition $D(J_n)$ is bounded above, so that $J_n^{>r}$ is finite for some large enough $r$, the statement (1) is satisfied. Similarly, if $n=m_{\max}$ and $r^{\max}_n=-\infty$, the statement (1) is satisfied, and we may thus assume from now on that $r^{\min}_m, r^{\max}_n\in\RR$.

We distinguish three cases.

\noindent
{\bf\underline{Case 1}:} $r^{\max}_n<r^{\min}_m$.

It then follows from Lemma~\ref{lemma:r+m} that $J_m^{<r}$ and $J_n^{>r}$ are both finite for any $r\in\RR$ with $r^{\max}_n<r<r^{\min}_m$. Hence (1) is satisfied in this case.

\noindent
{\bf\underline{Case 2}:} $r^{\max}_n=r^{\min}_m$.

Set $r=r^{\max}_n=r^{\min}_m$. We now prove that $J_m^{<r}$ and $J_n^{>r}$ must be both summable, so that (2) is satisfied.
By Lemma~\ref{lemma:r+m}, the sets $J_m^{<r-1/k}$ and $J_n^{>r+1/k}$ are finite for each $k\in\NN$, so that $J_m^{<r}=\bigcup_{k\in\NN}{J_m^{<r-1/k}}$ and $J_n^{>r}=\bigcup_{k\in\NN}{J_n^{>r+1/k}}$ are at most countable.

If $J_m^{<r}$ and $J_n^{>r}$ are both finite, there is nothing to prove. Assume now that at least one of $J_m^{<r}$ and $J_n^{>r}$ is infinite, say $J_m^{<r}$ (the other case being similar). Write $J_m^{<r}=\{i_1,i_2,\dots\}$. We distinguish two cases.

Assume first that $J_n^{>r}$ is infinite. Write $J_n^{>r}=\{j_1,j_2,\dots\}$.
For each $k\in\NN$, let $w_k\in W$ be the product of the transpositions $\tau_1,\dots,\tau_k$, where $\tau_s$ is the transposition exchanging $i_s$ and $j_s$ ($s\in\NN$). It then follows from (\ref{eqn:prod_transp}) that 
$$\la(w_k.\chi-\chi)=\sum_{s=1}^{k}{(\la_{i_s}-\la_{j_s})(d_{j_s}-d_{i_s})}=(m-n)\sum_{s=1}^{k}{(d_{j_s}-d_{i_s})}.$$
In particular,
$$\inf_{k\in\NN}\{\la(w_k.\chi-\chi)\}=-(n-m)\sum_{s=1}^{\infty}{(d_{j_s}-r+r-d_{i_s})}=-(n-m)(\Sigma(J_n^{>r})+\Sigma(J_m^{<r})).$$
Hence $J_m^{<r}$ and $J_n^{>r}$ must be both summable, as desired.

Assume next that $J_n^{>r}$ is finite. In particular $J_n^{>r}$ is summable, and it thus remains to show that $J_m^{<r}$ is also summable. Fix some $\epsilon>0$ and some sequence $(\epsilon_s)_{s\in\NN}$ of positive real numbers such that $\sum_{s\in\NN}{\epsilon_s}<\epsilon$. Since $r$ is an accumulation point for $J_n$, there is some infinite countable subset $\{j_1,j_2,\dots\}\subseteq J_n$ such that $d_{j_s}\geq\max(d_{i_s},r-\epsilon_s)$ for each $s\in\NN$. 
For each $k\in\NN$, let $w_k\in W$ be the product of the transpositions $\tau_1,\dots,\tau_k$, where $\tau_s$ is the transposition exchanging $i_s$ and $j_s$ ($s\in\NN$).  
It then follows from (\ref{eqn:prod_transp}) that
$$\la(w_k.\chi-\chi)=\sum_{s=1}^{k}{(\la_{i_s}-\la_{j_s})(d_{j_s}-d_{i_s})}=(m-n)\sum_{s=1}^{k}{(d_{j_s}-d_{i_s})}.$$
In particular,
$$\inf_{k\in\NN}\{\la(w_k.\chi-\chi)\}=-(n-m)\sum_{s=1}^{\infty}{(d_{j_s}-r+r-d_{i_s})}\leq -(n-m)(-\epsilon+\Sigma(J_m^{<r})).$$
Hence $J_m^{<r}$ must be summable, as desired.

\medskip
\noindent
{\bf\underline{Case 3}:} $r^{\max}_n>r^{\min}_m$.

Let $r\in\RR$ be such that $r^{\min}_m<r<r^{\max}_n$. Since $r^{\min}_m$ is an accumulation point for $J_m$, the set $J_m^{<r}$ is infinite. Similarly, since $r^{\max}_n$ is an accumulation point for $J_n$, the set $J_n^{>r}$ is infinite. Since there are infinitely many $r\in\RR$ with $r^{\min}_m<r<r^{\max}_n$, Lemma~\ref{lemma 3} then yields a contradiction in this case.
This concludes the proof of the proposition.
\end{proof}

\section{Characterisation of the PEC for $\la$ bounded and discrete}\label{subsection:charact_PEC_loc_fin}
In this section, we let $J$ denote some infinite set and $\la=(\la_j)_{j\in J}$ and $\chi=(d_j)_{j\in J}$ some tuples in $\RR^J$. 
In order to characterise the PEC for the triple $(J,\la,\chi)$, we will need to make some finiteness assumption on $\la$.
\begin{definition}\label{definition:discrete_bounded}
We call $\la$ \emph{bounded} (resp. \emph{discrete}) if the subset $\Lambda(J)$ of $\RR$ is bounded (resp. discrete).
\end{definition}

\begin{definition}\label{definition:cones_fin}
Given a set $J$ and a tuple $\la=(\la_j)_{j\in J}\in\RR^J$, we define the following cones in $\RR^J$:
\begin{equation*}
\begin{aligned}
C_{\min}(\lambda,A_J)&=\big\{(d_j)_{j\in J}\in \RR^J\ \big| \  \forall i,j\in J: \ \la_i<\la_j\implies  d_i\geq d_j\big\},\\
C_{\min}(\lambda,B_J)&=\big\{(d_j)_{j\in J}\in \RR^J \ \big| \ \forall j\in J: \ \la_jd_j\leq 0\quad\textrm{and}\quad \forall i,j\in J: \ |\la_i|<|\la_j|\implies  |d_i|\leq |d_j|\big\}.\\
\end{aligned}
\end{equation*}
We also define the sub-vector space $\ell^1(J)=\big\{(d_j)_{j\in J}\in\RR^J  \ \big| \ \sum_{j\in J}{|d_j|}<\infty\big\}$ of $\RR^J$. 
\end{definition}

\begin{prop}\label{prop:minimality}
Let $X\in\{A,B\}$ and set $W=W(X_J)$. For a triple $(J,\la,\chi)$, the following are equivalent:
\begin{enumerate}
\item
$\inf\la(W.\chi-\chi)=0$, i.e., $\lambda(w.\chi - \chi) \geq 0$ for all $w \in W$.
\item
$\chi\in C_{\min}(\lambda,X_J)$.
\end{enumerate}
\end{prop}
\begin{proof}
We first deal with the case $X=A$. If $\chi\notin C_{\min}(\lambda,A_J)$, then there exist some $i,j\in J$ with $\la_i<\la_j$ and $d_i<d_j$. Denoting by $w\in W$ the transposition exchanging $i$ and $j$, we deduce from (\ref{eqn:prod_transp}) that
$$\la(w.\chi-\chi)=(\la_{j}-\la_{i})(d_{i}-d_{j})<0,$$
so that $\inf\la(W.\chi-\chi)<0$. 

Assume conversely that $\chi\in C_{\min}(\lambda,A_J)$ and let $w\in W=S_{(J)}$. Let $I$ be some finite subset of $J$ containing the support of $w$. By assumption, we may then write $I=\{i_1,\dots,i_k\}$ so that $$\la_{i_1}\leq\dots\leq\la_{i_k}\quad\textrm{and}\quad d_{i_1}\geq\dots\geq d_{i_k}.$$ Together with (\ref{eqn:basic_loc_fin}), the rearrangement inequality then implies that 
$$\la(w\inv.\chi-\chi)=\sum_{s=1}^k{\la_{i_s}(d_{w(i_s)}-d_{i_s})}\geq \sum_{s=1}^k{\la_{i_s}(d_{i_s}-d_{i_s})}=0.$$
Hence $\inf\la(W.\chi-\chi)=0$, as desired.

We next deal with the case $X=B$. If $\chi\notin C_{\min}(\lambda,B_J)$, then either there exists some $j\in J$ such that $\la_jd_j>0$, or else $\la_kd_k\leq 0$ for all $k\in J$ and there exist some $i,j\in J$ with $|\la_i|<|\la_j|$ and $|d_i|>|d_j|$. In the first case, denoting by $\sigma\in W$ the element of $\{\pm 1\}^{(J)}$ with support $\{j\}$, we deduce from (\ref{eqn:basic_loc_fin}) that 
$$\la(\sigma.\chi-\chi)=-2\la_jd_j<0.$$
In the second case, denoting by $w\in S_{(J)}$ the transposition exchanging $i$ and $j$, and by $\sigma=(\sigma_k)_{k\in J}$ an element of $\{\pm 1\}^{(J)}$ with support in $\{i,j\}$ satisfying $\sigma_i\la_id_j=-|\la_id_j|$ and $\sigma_j\la_jd_i=-|\la_jd_i|$, we deduce from (\ref{eqn:basic_loc_fin}) that
$$\la(\sigma w\inv.\chi-\chi)=\la_i(\sigma_id_j-d_i)+\la_j(\sigma_jd_i-d_j)=(|\la_i|-|\la_j|)(|d_i|-|d_j|)<0.$$
In both cases, we deduce that $\inf\la(W.\chi-\chi)<0$.

Assume conversely that $\chi\in C_{\min}(\lambda,B_J)$ and let $w\in S_{(J)}$ and $\sigma=(\sigma_j)_{j\in J}\in\{\pm 1\}^{(J)}$. Let $I$ be some finite subset of $J$ containing the supports of $w$ and $\sigma$. By assumption, we may then write $I=\{i_1,\dots,i_k\}$ so that $$|\la_{i_1}|\leq\dots\leq |\la_{i_k}|\quad\textrm{and}\quad |d_{i_1}|\leq\dots\leq |d_{i_k}|.$$
Since moreover $\la_{i_s}d_{i_s}=-|\la_{i_s}d_{i_s}|$ for all $s=1,\dots,k$, we deduce from (\ref{eqn:basic_loc_fin}) and the rearrangement inequality that
$$\la(\sigma w\inv.\chi-\chi)=\sum_{s=1}^k{\la_{i_s}(\sigma_{i_s}d_{w(i_s)}-d_{i_s})}\geq \sum_{s=1}^k{|\la_{i_s}|(-|d_{w(i_s)}|+|d_{i_s}|)}\geq \sum_{s=1}^k{|\la_{i_s}|(-|d_{i_s}|+|d_{i_s}|)}=0.$$
Hence $\inf\la(W.\chi-\chi)=0$, as desired.
\end{proof}

\begin{lemma}\label{lemma:invariance_summable}
Let $X\in\{A,B\}$. Assume that $\la$ is bounded and that $\chi\in \ell^1(J)$. Then $(J,\la,\chi)$ satisfies the PEC for $W(X_J)$. 
\end{lemma}
\begin{proof}
Write $\la=(\la_j)_{j\in J}$ and $\chi=(d_j)_{j\in J}$, so that $C:=\sum_{j\in J}{|d_j|}<\infty$. Set $M:=\sup_{j\in J}{|\la_j|}<\infty$. Then for all $\sigma=(\sigma_j)_{j\in J}\in\{\pm 1\}^{(J)}$ and $w\in S_{(J)}$, we have
$$\la(\sigma w\inv.\chi-\chi)=\sum_{j\in J}{\la_j(\sigma_jd_{w(j)}-d_j)}\geq -M\sum_{j\in J}{(|d_{w(j)}|+|d_j|)}=-2MC.$$
Since $\sigma$ and $w$ were arbitrary, this proves the claim.
\end{proof}

\begin{lemma}\label{lemma:invariance_summable2}
Let $X\in\{A,B\}$. Assume that $\la$ is bounded. Then for all $\chi'\in \ell^1(J)$, the triple $(J,\la,\chi)$ satisfies the PEC for $W(X_J)$ if and only if the triple $(J,\la,\chi+\chi')$ satisfies the PEC for $W(X_J)$.
\end{lemma}
\begin{proof}
Set $W=W(X_J)$, and assume that $(J,\la,\chi)$ satisfies the PEC for $W$. Since
$$\inf\la\big(W.(\chi+\chi')-(\chi+\chi')\big)\geq \inf\la(W.\chi-\chi)+\inf\la(W.\chi'-\chi'),$$
the triple $(J,\la,\chi+\chi')$ then satisfies the PEC for $W$ by Lemma~\ref{lemma:invariance_summable}. Replacing $\chi$ by $\chi+\chi'$ and $\chi'$ by $-\chi'$, the converse follows.
\end{proof}

Note that, up to now, we only refered to the locally finite root systems of type $A_J$ and $B_J$. We now prove that these are indeed the only cases to be considered. 

\begin{lemma}\label{lemma:summability_at_0_B}
Assume that $(J,\la,\chi)$ satisfies the PEC for $W(B_J)$ or $W(D_J)$. Then
$$\sum_{j\in J_+}{|\la_jd_j|}<\infty, \quad\textrm{where $J_+:=\{j\in J \ | \ \la_jd_j>0\}$.}$$
\end{lemma}
\begin{proof}
Note first that since $W(D_J)=S_{(J)}\ltimes \{\pm 1\}_2^{(J)}\leq W(B_J)$, the PEC for $W(B_J)$ implies the PEC for $W(D_J)$. We may thus assume that $(J,\la,\chi)$ satisfies the PEC for $W(D_J)$. If $J_+$ is finite, there is nothing to prove. Assume now that $J_+$ is infinite. For each finite subset $I\subset J^+$ of even cardinality, let $\sigma_I\in \{\pm 1\}_2^{(J)}$ with support $I$. Then 
$$\inf\la(W(D_J).\chi-\chi)\leq \la(\sigma_I.\chi-\chi)=-2\sum_{i\in I}{\la_id_i}=-2\sum_{i\in I}{|\la_id_i|}.$$
Since $\sum_{j\in J^+}{|\la_jd_j|}$ is the supremum of all sums $\sum_{i\in I}{|\la_id_i|}$ with $I$ a finite subset of $J^+$ of even cardinality, the claim follows. 
\end{proof}

\begin{lemma}\label{lemma:D_J_condition_bis}
Assume that $(J,\la,\chi)$ satisfies the PEC for $W(D_J)$. Let $m,n\in\Lambda(J)$ with $|m|<|n|$. Then $D(J_{m})$ is bounded.
\end{lemma}
\begin{proof}
Assume for a contradiction that $D(J_{m})$ is unbounded, and choose some infinite countable set $\{i_0,i_1,\dots\}\subseteq J_{m}$ such that $|d_{i_s}|\geq s$ for all $s\in\NN$. Let also $j\in J_n$.
Fix some $s\in\NN$, and let $w\in S_{(J)}$ be the transposition exchanging $j$ and $i_s$.
Let also $\sigma=(\sigma_j)_{j\in J}\in\{\pm 1\}_2^{(J)}$ with support in $\{j,i_0,i_s\}$ be such that $m\sigma_{i_s}d_{j}=-|md_{j}|$ and $n\sigma_{j}d_{i_s}=-|nd_{i_s}|$. It then follows from (\ref{eqn:basic_loc_fin}) that 
\begin{equation*}
\begin{aligned}
\la(\sigma w\inv.\chi-\chi)&=m(\sigma_{i_0}d_{i_0}-d_{i_0})+m(\sigma_{i_s}d_{j}-d_{i_s})+n(\sigma_{j}d_{i_s}-d_{j})\\
&\leq 2|md_{i_0}|+|m|\cdot (-|d_{j}|+|d_{i_s}|)+|n|\cdot (-|d_{i_s}|+|d_{j}|)\\
&=2|md_{i_0}|-(|n|-|m|)(|d_{i_s}|-|d_{j}|)\\
&\leq 2|md_{i_0}|-(|n|-|m|)(s-|d_{j}|).\\
\end{aligned}
\end{equation*}
Hence $$\inf\la(W(D_J).\chi-\chi)\leq 2|md_{i_0}|-(|n|-|m|)(s-|d_{j}|).$$ As $s\in\NN$ was arbitrary, this contradicts the PEC for $W(D_J)$, as desired.
\end{proof}

\begin{lemma}\label{lemma:BCD}
Let $X\in\{B,C,D,BC\}$, and assume that $\la$ is bounded and discrete. Then $(J,\la,\chi)$ satisfies the PEC for $W(X_J)$ if and only if it satisfies the PEC for $W(B_J)$.
\end{lemma}
\begin{proof}
For $X=B,C,BC$, there is nothing to prove. Since $W(D_J)\leq W(B_J)$, it is also clear that the PEC for $W(B_J)$ implies the PEC for $W(D_J)$. Assume now that $(J,\la,\chi)$ satisfies the PEC for $W(D_J)=S_{(J)}\ltimes \{\pm 1\}_2^{(J)}$, and let us prove that it also satisfies the PEC for $W(B_J)=S_{(J)}\ltimes \{\pm 1\}^{(J)}$. 

If $\Lambda(J) = \{m\}$ is a one-element set and $(J,\la,\chi)$ satisfies the 
PEC for $W(D_J)$, then 
\[ \inf \lambda(W(D_J).\chi-\chi) 
= \inf_{\textrm{$|F|$ even}} \Big\{ - 2 m \sum_{j \in F} d_j \Big\}> -\infty \] 
implies 
\[ \inf \lambda(W(B_J).\chi-\chi) 
= \inf_{F} \Big\{ - 2 m \sum_{j \in F} d_j \Big\}> -\infty, \]
where the above infima are taken over finite subsets $F$ of $J$. 
We therefore assume that $|\Lambda(J)| \geq 2$.

We distinguish two cases. Assume first that $J$ has no accumulation point. Since $\Lambda(J)$ is finite, there exists some $m\in\Lambda^{\infty}(J)$. Thus $D(J_m)$ is unbounded. Lemma~\ref{lemma:D_J_condition_bis} then implies that $|m|\geq |n|$ for all $n\in\Lambda(J)$. Similarly, if $n\in\Lambda^{\infty}(J)$, then $D(J_n)$ is unbounded and hence $|n|=|m|$. In other words,
$\Lambda^{\infty}(J)\subseteq\{\pm m\}$ and $|m|=\max\{|m_{\min}|,|m_{\max}|\}$. Thus either $\Lambda^{\infty}(J)=\{m\}\subseteq\{m_{\min},m_{\max}\}$ with $|m|=\max_{j\in J}{|\la_j|}$, or else $m\neq 0$ and $\Lambda^{\infty}(J)=\{\pm m\}=\{m_{\min},m_{\max}\}$.

Since $(J,\la,\chi)$ is essentially bounded by Lemma~\ref{lemma:extremal_values} 
and $|\Lambda(J)| \geq 2$, the sets $D(J_{m_{\max}})$ and $D(J_{m_{\min}})$ are respectively bounded above and below. 
As $J$ has no accumulation point, this implies that the sets $J_{m_{\max}}^{>0}$ and $J_{m_{\min}}^{<0}$ are finite. Note also that the set $$S=\{j\in J \ | \ \la_j\neq m_{\min},m_{\max}\}$$ is finite because $\Lambda^{\infty}(J)\subseteq\{m_{\min},m_{\max}\}$ and $\Lambda(J)$ is finite. In particular, the tuple $\chi'=(d_j')_{j\in J}$ defined by 
\begin{equation*}
d_j'=
\left\{
\begin{array}{ll}
d_j\quad &\textrm{if $j\in S\cup J_{m_{\max}}^{>0}\cup J_{m_{\min}}^{<0}$,}\\
0\quad &\textrm{otherwise}
\end{array}
\right.
\end{equation*}
belongs to $\ell^1(J)$. Thus by Lemma~\ref{lemma:invariance_summable2}, we may replace without loss of generality $\chi$ by $\chi-\chi'$. In other words, we may assume that
$$D(J_{m_{\min}})\subseteq [0,\infty), \quad D(J_{m_{\max}})\subseteq (-\infty,0], \quad \textrm{and}\quad D(J_n)=\{0\} \quad \textrm{for all $n\neq m_{\min},m_{\max}$.}$$
If $J_{m_{\min}}$ (resp. $J_{m_{\max}}$) is finite, we may, by a similar argument, assume that $D(J_{m_{\min}})=\{0\}$ (resp. $D(J_{m_{\max}})=\{0\}$). On the other hand, if $J_{m_{\min}}$ (resp. $J_{m_{\max}}$) is infinite, then $D(J_{m_{\min}})$ (resp. $D(J_{m_{\max}})$) is unbounded by hypothesis, and hence $m_{\min}\leq 0$ (resp. $m_{\max}\geq 0$) by Lemma~\ref{lemma:summability_at_0_B}. 

If $\Lambda^{\infty}(J)=\{m\}$ with $|m|=\max_{j\in J}{|\la_j|}$, we may thus assume that $mD(J_m)\subseteq (-\infty,0]$ and that $D(J_n)=\{0\}$ for all $n\neq m$. If $\Lambda^{\infty}(J)=\{\pm m\}=\{m_{\min},m_{\max}\}$ is a $2$-element set, 
we may similarly assume that $\pm mD(J_{\pm m})\subseteq (-\infty,0]$ and that $D(J_n)=\{0\}$ for all $n\neq \pm m$. In both cases, $\chi\in C_{\min}(\lambda,B_J)$. Hence $(J,\la,\chi)$ satisfies the PEC for $W(B_J)$ by Proposition~\ref{prop:minimality}.

We next assume that $J$ has some accumulation point $r\in\RR$. Set $M=\sup_{j\in J}{|\la_j|}$. Let $\epsilon>0$ and let $S$ be an infinite subset of $J$ such that $D(S)\subseteq [r-\epsilon,r+\epsilon]$.
Let $\sigma=(\sigma_j)_{j\in J}\in \{\pm 1\}^{(J)}$ and $w\in S_{(J)}$, and let $I$ be some finite subset of $J$ containing the supports of $\sigma$ and $w$. Pick any $i\in S\setminus I$, and let $\tau=(\tau_j)_{j\in J}\in \{\pm 1\}^{(J)}$ with support contained in $\{i\}$ be such that $\tau\sigma\in \{\pm 1\}_2^{(J)}$. It then follows from (\ref{eqn:basic_loc_fin}) that
\begin{equation*}
\begin{aligned}
\la(\sigma w\inv.\chi-\chi)&=\sum_{j\in J}{\la_j(\sigma_{j}d_{w(j)}-d_j)}
=\sum_{j\in J}{\la_j(\tau_j\sigma_{j}d_{w(j)}-d_j)}-\la_i(\tau_id_i-d_i)\\
&\geq \la(\tau\sigma w\inv.\chi-\chi)-2M(|r|+\epsilon).
\end{aligned}
\end{equation*}
Hence $$\inf\la(W(B_J).\chi-\chi)\geq \inf\la(W(D_J).\chi-\chi)-2M(|r|+\epsilon).$$
This concludes the proof of the lemma.
\end{proof}

\begin{remark}\label{remark:min_energy_DJ}
For each $X\in\{A,B,C,D,BC\}$, let $C_{\min}(\lambda,X_J)$ denote the set of $\chi\in\RR^J$ such that $\inf\la(W(X_J).\chi-\chi)=0$. Note that, by Proposition~\ref{prop:minimality}, this is consistent with Definition~\ref{definition:cones_fin}. 

For $X\in\{C,BC\}$, we have $W(X_J)=W(B_J)$, and hence $C_{\min}(\lambda,X_J)=C_{\min}(\lambda,B_J)$. In order to determine $C_{\min}(\lambda,D_J)$, we associate to each
$\chi=(d_j)_{j\in J}\in\RR^J$ the (possibly empty) set $$I^{\mi}_{\la,\chi}:=\big\{i\in J \ \big| \ |\la_i|=\inf_{j\in J}|\la_j| \quad \textrm{and} \quad |d_i|=\inf_{j\in J}|d_j|\big\}.$$ 
Note that $C_{\min}(\lambda,B_J)\subseteq C_{\min}(\lambda,D_J)$ because $W(D_J)\leq W(B_J)$. One can then check, as in the proof of Proposition~\ref{prop:minimality} (or directly from \cite[Corollary~3.2]{Convexhull}), that $\chi\in C_{\min}(\lambda,D_J)$ if and only if either $\chi\in C_{\min}(\lambda,B_J)$, or else $\sigma_i\chi\in C_{\min}(\lambda,B_J)$ for some $i\in I^{\mi}_{\la,\chi}$, where $\sigma_i\in \{\pm 1\}^{(J)}$ has support $\{i\}$. As this fact will not be needed in our characterisation of the positive energy condition (see Lemma~\ref{lemma:BCD}), we leave it as an exercise.
\end{remark}

We first characterise the PEC for $W(A_J)$.

\begin{theorem}\label{theorem:charact_PEC_locally_finite}
Let $J$ be a set, and let $\la=(\la_j)_{j\in J}$ and $\chi=(d_j)_{j\in J}$ be elements of $\RR^J$. Assume that $\la$ is discrete and bounded. Then the following are equivalent:
\begin{enumerate}
\item
$(J,\la,\chi)$ satisfies the PEC for $W(A_J)$.
\item
$\chi\in C_{\min}(\lambda,A_J)+\ell^1(J)$.
\end{enumerate}
\end{theorem}
\begin{proof}
The implication (2)$\implies$(1) readily follows from Proposition~\ref{prop:minimality} and Lemma~\ref{lemma:invariance_summable2}. Assume now that $(J,\la,\chi)$ satisfies the PEC for $W(A_J)$, and let us prove that, up to substracting from $\chi$ some element of $\ell^1(J)$, one has $\chi\in C_{\min}(\lambda,A_J)$. Since $\chi\in C_{\min}(\lambda,A_J)$ if $\la$ is constant, we may assume that $(J,\la,\chi)$ is nontrivial, that is, $m_{\min}\neq m_{\max}$. Moreover, note that $(J,\la,\chi)$ is essentially bounded by Lemma~\ref{lemma:extremal_values}. 

By assumption, $\Lambda(J)$ is finite. Write $\overline{\Lambda^{\infty}}(J)=\Lambda^{\infty}(J)\cup\{m_{\min},m_{\max}\}=\{n_0,n_1,\dots,n_k\}$ so that $$m_{\min}=n_0<n_1<\dots<n_k=m_{\max}$$ for some $k\geq 1$. Proposition~\ref{prop:3cases} then implies that
$$r^{\min}_{n_0}\geq r^{\max}_{n_1}\geq r^{\min}_{n_1}\geq \dots\geq r^{\max}_{n_{k-1}}\geq r^{\min}_{n_{k-1}}\geq r^{\max}_{n_k}.$$
For each $t\in\{1,\dots,k\}$, we set
\begin{equation*}
a_t=
\left\{
\begin{array}{ll}
\tfrac{1}{2}(r_{n_{t-1}}^{\min}+r_{n_{t}}^{\max})\quad &\textrm{if $r_{n_{t-1}}^{\min},r_{n_{t}}^{\max}\in\RR$,}\\
r_{n_{t}}^{\max}+1\quad &\textrm{if $t=1$, $r_{n_{0}}^{\min}=\infty$ and $r_{n_{1}}^{\max}\in\RR$,}\\
r_{n_{t-1}}^{\min}-1\quad &\textrm{if $t=k$, $r_{n_{k}}^{\max}=-\infty$ and $r_{n_{k-1}}^{\min}\in\RR$,}\\
0\quad &\textrm{if $t=k=1$, $r_{n_{0}}^{\min}=\infty$ and $r_{n_{1}}^{\max}=-\infty$,}\\
\end{array}
\right.
\end{equation*}
so that 
$$r^{\min}_{n_0}\geq a_1\geq r^{\max}_{n_1}\geq r^{\min}_{n_1}\geq a_2\geq \dots\geq a_{k-1}\geq r^{\max}_{n_{k-1}}\geq r^{\min}_{n_{k-1}}\geq a_k\geq r^{\max}_{n_k}.$$
We also set $a_0:=\infty$ and $a_{k+1}:=-\infty$.
Fix some $t\in\{0,1,\dots,k\}$. We claim that the tuple $\chi_t=(d_j')_{j\in J}$ defined by 
\begin{equation*}
d_j'=
\left\{
\begin{aligned}
d_j-a_{t+1}\quad &\textrm{if $j\in J_{n_t}$ and $d_j<a_{t+1}$,}\\
d_j-a_t\quad &\textrm{if $j\in J_{n_t}$ and $d_j>a_t$,}\\
0\quad &\textrm{otherwise}
\end{aligned}
\right.
\end{equation*}
is in $\ell^1(J)$. 

Let $I_+$ (resp. $I_-$) denote the set of $j\in J_{n_t}$ such that $d_j>a_t$ (resp. $d_j<a_{t+1}$). We have to show that $$\sum_{j\in I_+}{|d_j-a_{t}|}<\infty\quad\textrm{and}\quad \sum_{j\in I_-}{|d_j-a_{t+1}|}<\infty.$$
We prove this for $I_+$, the proof for $I_-$ being similar. Since if $t=0$, the set $I_+$ is empty, we may assume that $t\in\{1,\dots,k\}$. Moreover, since $a_t\geq r^{\max}_{n_t}$, the set $I_+$ is finite as soon as $a_t> r^{\max}_{n_t}$. Note that this includes in particular the case where $t=k$ and $r_{n_{k}}^{\max}=-\infty$, in which case $D(J_{n_t})$ is bounded above and $J_{n_t}$ has no accumulation point, as well as the case where $t=1$, $r_{n_{0}}^{\min}=\infty$ and $r_{n_{1}}^{\max}\in\RR$, in which case $a_t=r^{\max}_{n_t}+1>r^{\max}_{n_t}$. Hence we may also assume that $a_t= r^{\max}_{n_t}\in\RR$ and that $r^{\min}_{n_{t-1}}\in\RR$. But then $r^{\min}_{n_{t-1}}=a_t=r^{\max}_{n_t}$, and hence the conclusion follows from Proposition~\ref{prop:3cases}.

Thus $\sum_{t=0}^k{\chi_t}\in \ell^1(J)$. Hence, up to replacing $\chi$ by $\chi-\sum_{t=0}^k{\chi_t}$, we may assume that 
\begin{equation}\label{eqn:intervals}
D(J_{n_t})\subseteq [a_{t+1},a_t] \quad\textrm{for all $t=0,1,\dots,k$}.
\end{equation}
We next define the tuple $\chi'=(d_j')_{j\in J}$ by 
\begin{equation*}
d_j'=
\left\{
\begin{array}{ll}
d_j-a_{t}\quad &\textrm{if $\la_j\in \Lambda(J)\setminus\overline{\Lambda^{\infty}}(J)$ and $n_{t-1}<\la_j<n_{t}$,}\\
0\quad &\textrm{otherwise.}
\end{array}
\right.
\end{equation*}
Since the set of $j\in J$ with $\la_j\in \Lambda(J)\setminus\overline{\Lambda^{\infty}}(J)$ is finite, $\chi'\in \ell^1(J)$. Moreover, it follows from (\ref{eqn:intervals}) that $\chi-\chi'\in C_{\min}(\lambda,A_J)$. This concludes the proof of the theorem.
\end{proof}

Before characterising the PEC for $W(B_J)$, we need one more lemma.
For a tuple $\nu=(\nu_j)_{j\in J}\in\RR^J$, we put
$$|\nu|:=(|\nu_j|)_{j\in J}\in\RR^J.$$

\begin{lemma}\label{lemma:PECB_PECAabs}
Assume that $(J,\la,\chi)$ satisfies the PEC for $W(B_J)$. Then $(J,-|\la|,|\chi|)$ satisfies the PEC for $W(A_J)$.
\end{lemma}
\begin{proof}
Let $w\in S_{(J)}$, and let $I$ be some finite subset of $J$ containing the support of $w$. Let $\sigma=(\sigma_j)_{j\in J}\in \{\pm 1\}^{(J)}$ with support in $I$ be such that $\la_{j}\sigma_{j}d_{w(j)}=-|\la_{j}d_{w(j)}|$ for all $j\in I$. It then follows from (\ref{eqn:basic_loc_fin}) that
$$-|\la|(w\inv.|\chi|-|\chi|)=-\sum_{j\in I}{|\la_{j}|\cdot (|d_{w(j)}|-|d_{j}|)}
\geq\sum_{j\in I}{\la_{j} (\sigma_jd_{w(j)}-d_{j})}=\la(\sigma w\inv.\chi-\chi).$$
Hence 
$$\inf\big(-|\la|(W(A_J).|\chi|-|\chi|)\big)\geq \inf\la(W(B_J).\chi-\chi),$$
as desired.
\end{proof}

\begin{theorem}\label{theorem:charact_PEC_locally_finiteB}
Let $J$ be a set, and let $\la=(\la_j)_{j\in J}$ and $\chi=(d_j)_{j\in J}$ be elements of $\RR^J$. Assume that $\la$ is discrete and bounded. Then the following are equivalent:
\begin{enumerate}
\item
$(J,\la,\chi)$ satisfies the PEC for $W(B_J)$.
\item
$\chi\in C_{\min}(\lambda,B_J)+\ell^1(J)$.
\end{enumerate}
\end{theorem}
\begin{proof}
The implication (2)$\implies$(1) readily follows from Proposition~\ref{prop:minimality} and Lemma~\ref{lemma:invariance_summable2}. Assume now that $(J,\la,\chi)$ satisfies the PEC for $W(B_J)$ and let us prove that, up to substracting from $\chi$ some element of $\ell^1(J)$, one has $\chi\in C_{\min}(\lambda,B_J)$. 

Since $(J,-|\la|,|\chi|)$ satisfies the PEC for $W(A_J)$ by Lemma~\ref{lemma:PECB_PECAabs}, we know from Theorem~\ref{theorem:charact_PEC_locally_finite} that $$|\chi|\in C_{\min}(-|\la|,A_J)+ \ell^1(J).$$ Let $\sigma=(\sigma_j)_{j\in J}\in\{\pm 1\}^J$ be such that $\sigma_jd_j\geq 0$ for all $j\in J$. In other words, $|\chi|=\sigma\chi$. Hence $\chi\in \sigma C_{\min}(-|\la|,A_J)+ \sigma \ell^1(J)$. Note that $\sigma \ell^1(J)=\ell^1(J)$. Up to substracting from $\chi$ some element of $\ell^1(J)$, we may thus assume without loss of generality that 
\begin{equation}\label{eqn:thmlocfin1}
\chi\in \sigma C_{\min}(-|\la|,A_J).
\end{equation}
Note that we may in addition assume that $(J,\la,\chi)$ satisfies the PEC for $W(B_J)$ by Lemma~\ref{lemma:invariance_summable2}.
We deduce from (\ref{eqn:thmlocfin1}) that $|\chi|\in C_{\min}(-|\la|,A_J)$, so that 
\begin{equation}\label{eqn:thmlocfin2}
\textrm{$\forall i,j\in J$: $|\la_i|<|\la_j|\implies  |d_i|\leq |d_j|$.}
\end{equation}
On the other hand, Lemma~\ref{lemma:summability_at_0_B} implies that 
$$\sum_{j\in J_+}{|d_j|}<\infty,\quad\textrm{where $J_+:=\{j\in J \ | \ \la_jd_j>0\}$.}$$
In particular, the tuple $\chi'=(d_j')_{j\in J}$ defined by 
\begin{equation*}
d_j'=
\left\{
\begin{array}{ll}
2d_j\quad &\textrm{if $j\in J_+$,}\\
0\quad &\textrm{otherwise}
\end{array}
\right.
\end{equation*}
belongs to $\ell^1(J)$. Since $\chi-\chi'=\sigma'\chi$, where $\sigma'=(\sigma'_j)_{j\in J}\in\{\pm 1\}^J$ has support $J_+$, the tuple $\chi-\chi'$ still satisfies the condition (\ref{eqn:thmlocfin2}), with $d_j$ replaced by $d_j-d'_j$ for all $j\in J$. Therefore, up to substracting $\chi'\in \ell^1(J)$ from $\chi$, we may assume that $\la_jd_j\leq 0$ for all $j\in J$ and that (\ref{eqn:thmlocfin2}) holds, so that $\chi\in C_{\min}(\la,B_J)$, as desired. 
\end{proof}

\section{Proof of Theorem~\ref{thm:mainintro}}\label{section:proof_main_thmintro}
By \cite{LN04}, the root system $\Delta$ of the locally finite split simple Lie algebra $\g$ over $\KK=\CC$ is isomorphic to one of the root systems $\Delta=\Delta(X_J)$ for $X\in\{A,B,C,D\}$ described in \S\ref{subsection:preliminaries}, where the Cartan subalgebra $\hh$ (resp. a one-dimensional extension of $\hh$ if $X=A$) is identified with $V_{\CC}:=V\otimes_{\RR}\CC=\sppan_{\CC}\{e_j \ | \ j\in J\}$. The identification of $V_{\CC}$ with $\hh^*$ (resp. a one-dimensional extension of $\hh^*$) induced by the assignment $e_j\mapsto \epsilon_j$ ($j\in J$) yields in turn an identification of the Weyl group $W\leq\GL(\hh^*)$ of $\g$ with the Weyl group $W(X_J)$ defined in \S\ref{subsection:Weyl_group_Delta}. 

Let $\la\in\hh^*$ be discrete and bounded. Then the restriction of $\la$ to $V$ is real valued, and $\la$ is determined by the tuple $(\la_j)_{j\in J}\in\RR^{J}$ defined by $\la_j=\la(e_j)$, $j\in J$. Identifying $\la$ with this tuple, $\la$ is then discrete and bounded in the sense of Definition~\ref{definition:discrete_bounded}.

Similarly, the character $\chi\co\ZZ[\Delta]\to\RR$ is the restriction of a $\ZZ$-linear map $$\widetilde{\chi}\co \sppan_{\ZZ}\{\epsilon_j \ | \ j\in J\}\to\RR:\epsilon_j\mapsto d_j,$$ and is thus determined by the tuple $(d_j)_{j\in J}\in\RR^{J}$. Note that $\ZZ[\Delta]=\sppan_{\ZZ}\{\epsilon_j \ | \ j\in J\}$ in all cases, except for $\Delta=\Delta(A_J)$, in which case $\ZZ[\Delta]$ is the corank $1$ submodule $\{\sum_{j\in J}{x_j\epsilon_j} \ | \ \sum_{j\in J}{x_j}=0\}$ of $\sppan_{\ZZ}\{\epsilon_j \ | \ j\in J\}$. Hence, either $\widetilde{\chi}=\chi$, or else $\Delta=\Delta(A_J)$ and $\widetilde{\chi}$ is determined by $\chi$ up to a constant. As $W(A_J)=S_{(J)}$, modifying $\widetilde{\chi}$ by a constant does not modify the value of the infimum of $\widetilde{\chi}(W(A_J).\la-\la)$. We may thus safely replace $\chi$ by $\widetilde{\chi}$, which we identify with the tuple $(d_j)_{j\in J}\in\RR^{J}$.

Finally, with the above identifications, we have for all $\sigma\in\{\pm 1\}^{(J)}$ and $w\in S_{(J)}$ that
\begin{equation*}
\begin{aligned}
\chi(\sigma w.\la-\la)&=\chi\Big(\sum_{j\in J}{\la_j(\sigma_{w(j)}\epsilon_{w(j)}-\epsilon_j)}\Big)=\sum_{j\in J}{\la_j(\sigma_{w(j)}d_{w(j)}-d_j)}=\sum_{j\in J}{d_j(\sigma_{j}\la_{w\inv(j)}-\la_j)}\\
&=\la\Big(\sum_{j\in J}{d_j(\sigma_{j}e_{w\inv(j)}-e_j)}\Big)=\la\Big(w\inv\sigma.\sum_{j\in J}{d_je_j}-\sum_{j\in J}{d_je_j}\Big)\\
&=\la((\sigma w)\inv.\chi-\chi),
\end{aligned}
\end{equation*}
and hence
$$\inf\chi(W.\la-\la)=\inf\la(W.\chi-\chi).$$
In particular, the condition $\inf\chi(W.\la-\la)>-\infty$ in the statement of Theorem~\ref{thm:mainintro} is equivalent to requiring the triple $(J,\la,\chi)$ with $\la=(\la_j)_{j\in J}$ and $\chi=(d_j)_{j\in J}$ to satisfy the PEC for $W=W(X_J)$ in the sense of Definition~\ref{definition:PEC}.

For $X=A$, Theorem~\ref{thm:mainintro} thus sums up Proposition~\ref{prop:minimality} and Theorem~\ref{theorem:charact_PEC_locally_finite}. Since $C_{\min}(\lambda,B_J)\subseteq C_{\min}(\lambda,D_J)$ (see Remark~\ref{remark:min_energy_DJ}) and since the PEC for $W(B_J)$, $W(C_J)$ and $W(D_J)$ are equivalent by Lemma~\ref{lemma:BCD}, the conclusion of Theorem~\ref{thm:mainintro} for $X\in\{B,C,D\}$ follows from Proposition~\ref{prop:minimality} and Theorem~\ref{theorem:charact_PEC_locally_finiteB}. \hspace{\fill}$\Box$

\bibliographystyle{amsalpha} 
\bibliography{these}

\end{document}